\documentclass{article}
\usepackage[utf8]{inputenc}

\usepackage{hyperref}
\usepackage{amsmath,amssymb,amsthm} 
\usepackage[space,noadjust,nocompress]{cite} 

\newcommand{\xad}{{x_{\alpha}^\delta}}

\newcommand{\yd}{{y_\delta}}
\newcommand{\pad}{{p_{\alpha}^\delta}}

\newcommand{\xd}{{x^\dagger}}

\newcommand{\xa}{{x_\alpha}}

\newcommand{\xads}{{x_{\alpha^*}^\delta}}
\newcommand{\xas}{{x_{\alpha^*}}}

\newcommand{\cur}{\gamma}
\newcommand{\MC}{{\rm MC}}
\newcommand{\Car}{C_{0}}
\newcommand{\Ci}{{C_1}}
\newcommand{\Cii}{{C_2}}
\newcommand{\as}{\alpha^*}
\newcommand{\ba}{\bar{\alpha}}
\newcommand{\uppera}{B}
\newcommand{\upperd}{V}

\newcommand{\pls}{\psi_{SL}}
\newcommand{\plr}{\psi_{SLR}}
\newcommand{\N}{{\mathbb N}}
\newcommand{\ignore}[1]{}

\newcommand{\rhi}{\omega}

\newcommand{\Id}{{\rm Id}}
\newcommand{\xl}{\chi} 
\newcommand{\yl}{\kappa}
\newcommand{\scp}[1]{\left\langle #1 \right \rangle}

 \newtheorem{lemma}{Lemma}
 \newtheorem{proposition}{Proposition}
 \newtheorem{definition}{Definition}
 \newtheorem{theorem}{Theorem}
 \newtheorem{remark}{Remark}

\hypersetup{%
pdftitle={A simplified L-curve method as error estimator},
pdfauthor={Stefan Kindermann},
pdfkeywords={heuristic parameter choice, L-curve method, regularization}
}
\title{A simplified L-curve method as error estimator}
\author{Stefan Kindermann\footnotemark[2]
\and Kemal Raik\footnotemark[3]
}

\begin{document}
\maketitle
\renewcommand{\thefootnote}{\fnsymbol{footnote}}
\footnotetext[2]{Industrial Mathematics Institute, Johannes Kepler University Linz, Altenbergerstra{\ss}e 69, 4040 Linz, Austria 
({\tt kindermann@indmath.uni-linz.ac.at}).}
\footnotetext[3]{Industrial Mathematics Institute, Johannes Kepler University Linz, Altenbergerstra{\ss}e 69, 4040 Linz, Austria 
({\tt kemal.raik@indmath.uni-linz.ac.at}).}

\begin{abstract}
The L-curve method is a well-known heuristic  method for choosing the regularization parameter for ill-posed problems by selecting 
it according to the maximal curvature of the L-curve. In this article, we propose a simplified version that replaces 
the curvature essentially by the derivative of the parameterization on the $y$-axis. This method shows a similar 
behaviour to the original L-curve method, but unlike the latter, it may serve as an error estimator under 
typical conditions. Thus, we can accordingly prove convergence for the simplified L-curve method. 
\end{abstract}

\section{Introduction}
The L-curve criterion is one of the best-known heuristic methods for choosing 
the regularization parameter in various regularization methods for ill-posed 
problems. 
One of the first instances of an 
L-curve graph appeared in the book by Lawson and Hanson~\cite{LaHa74}, 
although it was not related to a parameter choice procedure. 
That it can be the  basis for a parameter choice method was 
suggested by Hansen and O'Leary~\cite{HaLe93} and 
 further analyzed and popularized by Hansen \cite{Ha92}.  

The methodology is well-known: Suppose that we are faced with the problem of 
solving an ill-posed problem of the form 
\begin{equation}\label{main} A x = y, \end{equation}
where $y$ are data and $A:X\to Y$ is a continuous linear operator between Hilbert spaces 
which lacks a continuous inverse. Moreover, we assume that only noisy data 
\[ \yd = y + e, \qquad \|e\|\leq \delta, \quad y = A \xd,  \]
are available, where $\xd$ denotes the "true" unknown solution (or, more 
precisely, the minimal-norm solution). Here, $e$ denotes an unknown error, and 
its norm is called the noise-level $\delta$. In the case of heuristic parameter 
choice rules, which the L-curve method is an example of, this noise-level is considered 
unavailable. 

As the inverse of $A$ is not bounded, the problem \eqref{main} cannot be solved 
by classical inversion algorithms, rather, a regularization scheme has to be 
applied \cite{EnHaNe96}. That is, one constructs a one-parametric family of continuous 
operators $(R_\alpha)_\alpha$, with $\alpha >0$, that in some sense approximates the 
inverse of $A$ for $\alpha \to 0$.

An approximation  to the true solution of \eqref{main},
denoted as $\xad$, is computed by means of the regularization operators: 
\[ \xad = R_\alpha \yd. \]
A delicate issue in regularization schemes is the choice of the regularization 
parameter $\alpha$ and the standard methods make use of the noise-level $\delta$. 
However, in situations when this is not available, so-called {\em heuristic} parameter 
choice methods \cite{KiETNA} are proposed.
The L-curve method selects an $\alpha$ corresponding to the  corner point  of 
the graph $(\log(\|A\xad -\yd\|),\log(\|\xad\|))$ parameterized by $\alpha$.

Recently,  \cite{KiNe, KiETNA}
a convergence theory for certain heuristic parameter choice rules was developed. Essential in this analysis is a restriction on the noise that rules out noise which is "too regular". Such noise conditions in the form of Muckenhoupt-type conditions 
were used in \cite{KiNe,KiETNA} and are currently the standard tool in the analysis of heuristic 
rules. If these conditions hold, then several well-known heuristic parameter choice rules serve as 
error estimators for the total error in typical regularization schemes and convergence and convergence rate results follow. 

The L-curve method, however, does not seem to be accessible to such an analysis, 
although some of its properties were investigated, for instance, by Hansen \cite{Ha92,Ha01} 
and Reginska \cite{Re96}.
Nevertheless, it does not appear that it can be related to some sort of error estimators directly. 

There are various suggestions for efficient practical implementations of the L-curve method, 
like Krylov-space methods \cite{LoRo,CaHaRe} or model functions \cite{LuMa}. 
Note that the method is also implemented in Hansen's Regularization Tools \cite{HaReg}. 
A generalization of the L-curve method in form of the Q-curve method was recently 
suggested by Raus and H\"amarik \cite{RaHa}. Other simplifications or variations 
are the V-curve \cite{Vcurve} or the U-curve \cite{Ucurve}. Some overview and comparisons 
of  other heuristic and non-heuristic methods are given in \cite{HaPaRa,HaPaRa11,BaLu} and 
the PhD.~thesis of Palm \cite{Pa10}.

The aim of this article is to propose a simplified version of the L-curve method by dropping 
several terms in the expression for the curvature of the L-graph. We argue that this simplified 
version does not alter the original method significantly, and, moreover, we prove that 
the simplified L-curve has error estimating capabilities similar to several other well-known 
heuristic methods. This allows us to state conditions under which we can verify 
convergence of the simplified L-curve method.

\subsection{The L-Curve method and its simplification for Tikhonov regularization} 
We use a standard setting of an ill-posed problem of the form \eqref{main}. Although not 
necessary for our analysis and only used for clarity, we assume that $A$ is a compact 
operator, which then has a singular value decomposition (SVD) $(\sigma_i,u_i,v_i)_{i \in \N}$, with the 
positive singular values $\sigma_i$ and the singular functions $u_i \in X$, $v_i \in Y$ such that 
\[ A x = \sum_i \sigma_i \scp{u_i,x} v_i, \qquad \lambda_i:= \sigma_i^2>0,\]
where $\scp{\cdot,\cdot}$ denotes the scalar product in $X$ (or also in $Y$). 
As regularization operator, we employ
Tikhonov regularization, which 
defines a regularized solution  to \eqref{main} (an approximation to the true solution $\xd$) via 
\begin{equation}\label{Tikreg}
\xad  := (A^*A + \alpha I)^{-1} A^* \yd. 
\end{equation}
Here $\alpha\in(0,\alpha_{\text{max}})$ is the regularization parameter. 
For notational purposes we also define the (negative) residual $\pad$ 
and the auxiliary regularized solution with exact data  $\xa$ 
\begin{align}
\pad &:= \yd - A \xad = \alpha (A A^* +\alpha I)^{-1} \yd, \label{res} \\
\xa  &:= (A^*A + \alpha I)^{-1} A^* y. \label{appx}
\end{align} 
The overall goal of a good parameter choice is always to 
minimize the total error $\|\xad -\xd\|$, which can be 
bounded by the sum of 
the {\em stability error}  $\|\xad-\xa\|$ and 
the {\em approximation error} $\|\xa -\xd\|$:
\begin{equation}\label{upper} \|\xad -\xd\| \leq \|\xad -\xa\| + \|\xa - \xd\|. \end{equation}
It is well-known that the approximation error can in general decay arbitrarily slowly. 
In order to establish bounds for it and thus derive convergence rates, one has to postulate 
a certain smoothness condition on $\xd$ in the form of a source condition:
Here, we focus on H\"older source conditions, i.e., such a source condition holds if $\xd$ can be expressed as 
\begin{equation}\label{source}
\xd = (A^*A)^\mu \omega, \qquad \|\omega\| \leq C, \qquad \mu > 0. 
\end{equation} 
In terms of the SVD, $\xd$ satisfies \eqref{source} if 
\[ \sum_i \frac{|\scp{\xd,u_i}|^2}{\lambda_i^{2\mu}} < \infty. \]
If this is the case, then for Tikhonov regularization we have that 
\begin{equation}\label{eq:conratest}  \|\xa - \xd\| \leq C \alpha^{\mu}, \qquad \text{ for } 0 \leq \mu \leq 1, \end{equation}
and the convergence rate
\[  \|\xad -\xd\| \leq C {\delta^\frac{2\mu}{2\mu+1}}, \qquad \text{ for } 0 \leq \mu \leq 1,  \]
which is known to be the optimal order of convergence under \eqref{source}. 
We also note the saturation effect of Tikhonov regularization, which means that the rates do not 
improve for higher source conditions beyond $\mu >1$; see, e.g.,  \cite{EnHaNe96}.

\subsection{The L-curve} 
The L-curve is a plot of the (logarithm of the) residual against the 
(logarithm of the) norm of the regularized solution. 
Define the following curve parameterized by the regularization parameter $\alpha$ 
\[  \yl(\alpha) = \log(\|\pad\|^2) = \log(\|A \xad - \yd\|^2),\qquad 
\xl(\alpha) =  \log(\|\xad\|^2).  \]
Then a plot of the curve 
\begin{equation}\label{afist} \alpha \to \begin{bmatrix}  \yl(\alpha)\\
\xl(\alpha) \end{bmatrix}, 
\end{equation}
yields a graph, which often resembles the shape of an "L", hence its name  L-curve. The idea of the L-curve method is to choose $\alpha$ as the 
curve parameter that corresponds to the {\em corner point} of the "L". 
Since a corner has a large curvature, the operational definition of 
the parameter selection by the L-curve is that of the maximizer (over the selected range of $\alpha$) 
of the curvature of the L-graph, i.e., $\alpha=:\as$ is selected as 
\[ \as = \text{argmax}_\alpha \,\,  \cur(\alpha), \] 
with the signed curvature defined as  (see, e.g.,  \cite{Ha01}), 
\[  \cur(\alpha) = \frac{\xl''(\alpha) \yl'(\alpha)-\xl'(\alpha) \yl''(\alpha) }{\left( \xl'(\alpha)^2+ \yl'(\alpha)^2\right)^{3/2} }. \] 
Here  a prime $'$ denotes differentiation with respect to $\alpha.$
For Tikhonov regularization and many other methods, 
it is not difficult to realize that $\yl(\alpha)$ is strictly monotonically decreasing in $\alpha$, 
hence, the L-curve can be considered as a graph of a function $f = \xl(\yl^{-1})$.


As already observed by Hansen \cite{Ha01}, for Tikhonov regularization
the curvature  does not involve second derivatives and can be reduced to 
\begin{equation}\label{form} 
\cur(\alpha) 
= \frac{\eta \rho}{|\eta'|}
\frac{\rho \eta + \alpha \eta' \rho +
\alpha^2 \eta' \eta }{\left(\rho^2 +\alpha^2 \eta^2\right)^\frac{3}{2}},  
\end{equation} 
where 
\begin{equation}
 \eta = \eta(\alpha):= \|\xad\|^2, \qquad \rho = \rho(\alpha):= \|\pad\|^2. 
 \end{equation}
The following lemma investigates this expression:
\begin{lemma}\label{firstlem}
We have 
 \begin{align}
\cur(\alpha) 
&= \frac{\eta }{\alpha |\eta'|}
\frac{\zeta^2 }{\left(\zeta^2 +1 \right)^\frac{3}{2}}-  \frac{\zeta(1 +\zeta)}{\left(\zeta^2 +1 \right)^\frac{3}{2}}
, \qquad \zeta = \zeta(\alpha) := \frac{\rho}{\alpha \eta} \label{formsi} \\
&=: \frac{\eta }{\alpha |\eta'|} c_1(\zeta)  - c_2(\zeta),  \label{formsi2} 
\end{align} 
where 
\[ 0\leq c_1(\zeta) \leq \frac{2}{3 \sqrt{3}}, \qquad 
0 \leq c_2(\zeta) \leq \frac{1}{\sqrt{2}}. \]
\end{lemma}
\begin{proof}
The expression \eqref{form} can easily be rewritten as \eqref{formsi} with 
\[ c_1(\zeta)= \frac{\zeta^2 }{\left(\zeta^2 +1 \right)^\frac{3}{2}},\qquad 
c_2(\zeta) =  \frac{\zeta(1 +\zeta)}{\left(\zeta^2 +1 \right)^\frac{3}{2}}.\] 
By elementary calculus, we may find the maxima for $c_1$ at $\zeta=\sqrt{2}$ and for $c_2$ at  $\zeta = 1$ yielding  
the upper bounds. 
\end{proof}

According to the rationale for the L-curve method, we are searching for a corner of the L-graph, i.e., by definition a point where 
$\cur(\alpha)$ has a  {\em large positive} value . 
(An ideal corner has infinite curvature.)
Thus, according to \eqref{formsi2}, the only expression in the previous 
lemma that could contribute to large 
values is 
$\frac{\eta}{\alpha |\eta'|}$. 
Hence, backed by Lemma~\ref{firstlem}, we 
propose to remove the $\zeta$-dependent expressions and instead of \eqref{afist}, 
maximize the functional 
\[ \as = \text{argmax}_\alpha \,  \frac{\eta}{\alpha |\eta'|},  \]
which leads to the   simplified 
L-curve methods of this article. Instead of maximization we may equivalently consider minimizing the reciprocal. Moreover, we propose two versions of the simplified method
(the factor $\frac{1}{2}$ below is introduced for notational purposes and is irrelevant for the analysis and the method): 
\begin{definition}
The simple-L method selects the regularization parameter $\alpha$ as the minimizer (over a range of $\alpha$-values)  of 
the simple-L functional:
\begin{equation}
\begin{split} 
\as &= \text{\rm argmin}_\alpha \, \pls(\alpha), \\ 
\pls(\alpha)&:= \left(-\frac{1}{2}\alpha  \eta'(\alpha) \right)^\frac{1}{2} = \left(-
\scp{\xad, \alpha \frac{\partial}{\partial \alpha} \xad }\right)^\frac{1}{2} . 
\end{split} 
\end{equation} 
The simple-L ratio method selects $\alpha$ as minimizer (over a range of $\alpha$-values) of 
\begin{equation}
\begin{split} 
\as &= \text{\rm argmin}_\alpha\, \plr(\alpha), \\ 
\plr(\alpha)&:=  \left(-\frac{1}{2}\alpha \frac{\eta'(\alpha)}{\eta(\alpha)}\right)^\frac{1}{2}  =  \left(\frac{-\scp{ \xad, \alpha \frac{\partial}{\partial \alpha} \xad }}{\|\xad\|^2} \right)^\frac{1}{2} . 
\end{split}
\end{equation} 
\end{definition}
The main advantage that these simplified L-curve methods hold is that under certain conditions, 
they serve as error estimators and convergence of the associated parameter choice methods 
can be proven in contrast to the original L-curve method. 

Another reason for using the simplified functionals is that $\pls$ resembles and can be compared with 
several other heuristic parameter choice functionals, which are known to have an error-estimating property. 
For instance the quasi-optimality  (QO) principle defines $\alpha$ as the minimizer  of 
\[ \psi_{QO}(\alpha) := \left\|\alpha \frac{\partial}{\partial \alpha} \xad\right\|, \]
while the heuristic discrepancy (HD) principle define it as minimizer of 
\[ \psi_{HD}(\alpha) := \frac{\left\| \pad\right\|}{\sqrt{\alpha}}. \]
An improvement of the HD-rule is the Hanke-Raus (HR) rule, which is defined as 
\[ \psi_{HR}(\alpha) := \left( \frac{1}{\alpha} \scp{\pad,\pad^{II}} \right)^\frac{1}{2}, \]
where $\pad^{II}$ is the second Tikhonov iterate;
for details, see, e.g., \cite{KiETNA}. 
For Tikhonov regularization, these $\psi$-functionals can be written in terms of the singular value decomposition
as 
\[ \psi(\alpha)^2 = \sum_{i} \frac{\alpha^{n-k-1} \lambda_i^k}{(\alpha + \lambda_i)^n} |\scp{\yd,v_i}|^2, \]
with $n = 4$, $k = 1$ for QO, $n = 3$, $k = 0$ for HR, and $n = 2$, $k = 0$ for HD. 
The structural similarity of these methods has led Raus to define the so-called R1-family of 
rules \cite{RausR1}, which QO and HR are special instances of. 

In terms of singular values, the $\pls$ functional can be written as 
\[ \pls(\alpha)^2 = \sum_{i} \frac{\alpha \lambda_i}{(\alpha + \lambda_i)^3} |\scp{\yd,v_i}|^2, \]
and we observe that it follows a similar pattern as the others with $n = 3$ and $k = 1$. 
Note, however, that it does not fall into the R1-class of rules. 

As for the other rules (see, e.g., \cite{Ne08,KiETNA}), one may also extend the definition of $\pls$ (and $\plr$) to more 
general regularization schemes: If $R_\alpha$ is defined by a filter function 
$g_\alpha(\lambda)$ of the form 
\[ R_\alpha \yd := g_\alpha(A^*A) A^* \yd, \qquad \text{ with } \quad 
 r_\alpha(\lambda):= 1 - \lambda g_\alpha(\lambda),  \]  
then we may extend the definition of $\pls$ as 
\begin{equation}\label{xy}
\pls(\alpha) = \| \rho_\alpha(A^*A)^\frac{1}{2} \yd\|, \qquad \rho_\alpha(\lambda) = \lambda g_\alpha(\lambda)^2 |r_\alpha(\lambda)|. \
\end{equation}
This definition agrees with that for Tikhonov regularization, where $g_\alpha(\lambda) = \frac{1}{\alpha +\lambda}$
and 
\begin{equation}\label{xxa}\rho_\alpha(\lambda) = \frac{\alpha \lambda_i}{(\alpha + \lambda_i)^3}. \end{equation}

\begin{remark} \rm
Let us also mention that the simple-L and simple-L ratio methods have some similarities with the V-curve method \cite{Vcurve}, 
which is defined as minimizer of the speed of the  parameterization of the L-curve on a logarithmic grid. 
Thus, the minimization functional for the $V$-curve is for Tikhonov regularization (using the identity $\rho' = -\alpha \eta'$; cf.~\cite{Ha01})
\begin{align*} 
\psi_{V}(\alpha)  &= \left\| \begin{bmatrix} \alpha \yl'(\alpha) \\ \alpha \xl'(\alpha) \end{bmatrix} \right\|  =
 \left\| \begin{bmatrix} \alpha  \frac{\rho'}{\rho} \\
\alpha  \frac{\eta'}{\eta}  \end{bmatrix} \right\|  \\
& = \alpha |\eta'| \sqrt{\frac{\alpha^2}{\rho^2}+\frac{1}{\eta^2}} = 
\pls(\alpha)^2 \sqrt{\frac{\alpha^2}{\rho^2}+\frac{1}{\eta^2} } \\
& =\plr(\alpha)^2  \sqrt{\frac{\alpha^2|\eta|^2}{\rho^2}+1 } = 
\plr(\alpha)^2  \sqrt{\frac{1}{\zeta^2}+1 }. 
\end{align*} 
Thus, the V-curve is essentially a weighted form (with weight 
$\sqrt{\frac{1}{\zeta^2}+1 }\geq 1$)
of our simple L-ratio functional $\plr$. 
It is obvious that the simple-L functional equals  the derivative of the parameterization of the 
y-axis of the non-logarithmic L-curve 
$(\rho(\alpha),\eta(\alpha))$ weighted with $\alpha$, which also equals
the derivative of the x-axis parameterization as $\rho'(\alpha) = -\alpha \eta'$.  


Another related method is the 
so-called composite residual and smoothing operator method (CRESO-method) \cite{creso85}.  It defines 
the regularization parameter by an 
argmax of the 
function 
\[ C(\alpha):= \|\xad\|^2 + 2 \alpha \frac{\partial}{\partial \alpha } \|\xad\|^2 = \eta + 2 \alpha \eta'.\]
Since maximizing $C(\alpha)$ is the same as 
minimizing $-C(\alpha)$,  we observe that the method minimizes the functional
\[ -C(\alpha) = \eta(\alpha) (2\plr(\alpha,\yd)^2 -1).\]
Since $\eta(\alpha)$ is bounded from below
(and approaches $\|\xd\|^2$ for the optimal 
choice of $\alpha$), we 
may regard the CRESO method essentially as 
a variant of the simple-L ratio method.

It is  worth mentioning that the expression denoted by $\zeta$ in the curvature in Lemma~\ref{firstlem} 
also has a relation to existing  parameter choice functionals.  In fact, in the simplest case, the 
Brezinski-Rodriguez-Seatzu rule \cite{BrRoSe08,BRS2} is defined as the minimizer of 
$ \frac{\|A \xad -\yd\|^2}{\alpha \|\xad\| },$ which in our notation equals $\|\xad\| \zeta$. 

\end{remark}

\section{Convergence theory for Tikhonov regularization} 
The convergence theory for error-estimating heuristic methods is 
based on the idea that such a functional $\psi(\alpha)$ behaves 
in a similar way to the total error $\|\xad-\xd\|$. Hence, minimizing 
$\psi(\alpha)$ should also give a small total error and thus a 
good parameter choice.  For verifying this,  we have 
to estimate the functionals against the approximation and stability errors, 
which can be expressed in terms of the SVD  as follows: 
\begin{align}  
\|\xad -\xa\|^2 &= \sum_{i} \frac{\lambda_i}{(\lambda_i +\alpha)^2} |\scp{\yd-y,v_i}|^2, \\
\|\xa -\xd\|^2 &= 
\sum_i  \frac{\alpha^2}{(\lambda_i +\alpha)^2}   |\scp{\xd,u_i}|^2. 
\end{align}
As usual, the total error $\|\xad -\xd\|$ can be bounded 
by the stability error and the approximation error as in \eqref{upper}. 

%
We may split  the functional $\pls$ in a similar way, into a  noise-dependent term
and an $\xd$-dependent one: 
\begin{align*}  \pls(\alpha,e)^2&:=  \|\rho_\alpha(AA^*)^\frac{1}{2}(\yd-y)\|^2 = 
 \sum_{i} \frac{\alpha \lambda_i}{(\lambda_i +\alpha)^3} |\scp{\yd-y,v_i}|^2, \\
 \pls(\alpha,\xd)^2& :=  \|\rho_\alpha(AA^*)^\frac{1}{2} y\|^2 = 
   \|\rho_\alpha(AA^*)^\frac{1}{2} A \xd \|^2 = 
  \sum_{i} \frac{\alpha \lambda_i^2 }{(\lambda_i +\alpha)^3} |\scp{\xd,u_i}|^2.  
 \end{align*} 
Obviously, we have the bound 
\begin{equation}\label{eq:psisplit}
    \pls(\alpha) \leq \pls(\alpha,e) +  \pls(\alpha,\xd).\end{equation}
Convergence is based on the following theorem which is proven in \cite{KiETNA}:
\begin{theorem}\label{conth}
Let a $\psi$-functional be given by some nonegative continuous function
$\rho_\alpha(\lambda)$ defined on the spectrum of $A^*A$.
Let $\as$ be selected as 
$$\as = \text{\rm argmin}_\alpha\, \|\rho_\alpha(A^*A)^\frac{1}{2} \yd\|.$$ 
Assume that 
\begin{align}\label{eq:upup}
\|\rho_\alpha(A^*A)^\frac{1}{2} A \xd \| &\leq  \uppera(\alpha), & 
\|\rho_\alpha(A^*A)^\frac{1}{2} (\yd-y)\|&\leq \upperd(\alpha), 
 \end{align} 
 where $ \uppera(\alpha)$ 
is monotonically increasing and 
$ \upperd(\alpha)$ is monotonically decreasing. 
Furthermore, assume the following lower bounds involving the 
stability and approximation errors: 
\begin{align} 
\|\xad -\xa\| & \leq  \Car \|\rho_\alpha(A^*A)^\frac{1}{2} (\yd-y)\|, \label{esta}\\  
\|\xa -\xd\| &\leq  \Phi\left(\|\rho_\alpha(A^*A)^\frac{1}{2} A \xd \|\right), \label{estd}
\end{align} 
with some increasing  function $\Phi$. 
Then the total error can be bounded by 
\[ \|\xads - \xd\| \leq \Phi\left( 2 \inf_\alpha \left\{  \uppera(\alpha)  +  \upperd(\alpha)\right\} \right) + 
2 \Car \inf_\alpha \left\{  \uppera(\alpha)  +  \upperd(\alpha) \right\}.\] 
\end{theorem} 

If a heuristic parameter choice functional 
$\psi$ is (under certain circumstances) a good estimator for both the approximation error and the stability error,
i.e., both the lower bounds hold and the upper bounds $\uppera,\upperd$ are close to the approximation and stability error, then  
the  corresponding parameter choice 
is usually a successful one in the sense that it 
yields the optimal order of convergence.

\subsection{\boldmath Upper bounds for $\pls$} 
At first we provide upper bound  for   $\pls(\alpha,\yd-y)$:
Since $\rho_\alpha(\lambda) \leq \frac{\lambda}{(\lambda +\alpha)^2}$, the next 
result follows 
immediately:
\begin{lemma}
We have that 
\begin{equation}\label{eq:slup}
\pls(\alpha,e) \leq \upperd(\alpha):= \|\xad -\xa\| \leq  \frac{\delta}{\sqrt{\alpha} }. \end{equation}
\end{lemma} 
The term $\pls(\alpha,y)$ can be bounded in the 
following way:
\begin{lemma} 
We have 
\begin{equation}\label{eq:slap}
 \pls(\alpha,\xd)   \leq \uppera(\alpha):= \left( \sum_i 
\frac{\alpha }{(\lambda_i +\alpha)} |\scp{u_i,\xd}|^2\right)^\frac{1}{2} = 
\scp{\xd-\xa,\xd}^\frac{1}{2},
\end{equation} 
and $\uppera(\alpha)$ is monotonically increasing in $\alpha$. 
Moreover, if a source condition \eqref{source} is satisfied, then 
\[ \uppera(\alpha) \leq C \alpha^\mu \quad \text{for } \mu \leq \frac{1}{2}. \] 
\end{lemma} 
\begin{proof} 
Noting the definition of $\rho_\alpha(\lambda)$ in \eqref{xxa} and that 
$\frac{\lambda}{(\lambda +\alpha)} \leq 1$, we have that 
$\rho_\alpha(\lambda)^2 \lambda \leq \frac{\alpha}{(\lambda +\alpha)}$, 
which verifies the result. 
The fact that the last expression is monotone and allows for convergence rates is standard. 
\end{proof} 

\begin{remark}\label{earlysat} \rm
From the previous lemmas we obtain that under a 
source condition and by~\eqref{eq:psisplit}
\[ \inf_\alpha \pls(\alpha) \leq 
\inf_\alpha \left(\uppera(\alpha) + \upperd(\alpha) \right)\leq 
\inf_\alpha \left(C \alpha^\mu  + \frac{\delta}{\sqrt{\alpha}}\right)
\sim \delta^\frac{2\mu}{2\mu+1}, \qquad 
\text{for } \mu \leq \frac{1}{2}. \]
This is the optimal-order rate of the error, but it is only 
achieved under the restriction that  $\mu \leq \frac{1}{2}$. 
Thus, $\pls$ shows {\em early saturation}, that is, 
it is only of the same order as the optimal rate 
for a lower smoothness index, 
but it shows suboptimal rates for $\mu \geq \frac{1}{2}$. This is akin to 
the early saturation of the discrepancy principle~\cite{EnHaNe96} 
and the HD-method \cite{KiETNA}.
\end{remark}

\subsection{\boldmath Lower bounds for $\pls$} 
The main 
issue in the convergence theory is to find conditions 
which  are sufficient to verify the lower bounds in Theorem~\ref{conth}. 
However, it is well-known that due to the so-called 
Bakushinskii veto \cite{Ba84,KiETNA}, a heuristic parameter choice functional {\em cannot}  be a valid estimator for the 
stability error in the sense that \eqref{esta} holds 
unless the permissible noise $\yd-y$ is  restricted  in some sense. 
Conditions imposing such 
noise restrictions are at the heart of the convergence theory.

We recall the following classical noise restrictions 
that were used in \cite{KiNe,KiETNA} denoted as  
Muckenhoupt-type conditions (MC): 

\begin{definition} 
The condition $\MC_1$ is satisfied if 
there exists a constant $\Ci$ such that for all 
appearing errors $e =\yd-y$ it holds that  for all $0 \leq \alpha \leq \alpha_{max}$, 
\begin{equation}\label{eq:MC1} 
\sum_{\lambda_i \geq \alpha}  \frac{\alpha}{\lambda_i} |\scp{e,v_i}|   ^2  
\leq \Ci 
\sum_{\lambda_i \leq \alpha}  |\scp{e,v_i}|   ^2.  
\end{equation} 
The condition $\MC_2$ is satisfied if 
there exists a constant $\Cii$ such that for all 
appearing errors $e =\yd-y$ it holds that   for all $0 \leq \alpha \leq \alpha_{max}$, 
\begin{equation} \label{eq:MC2} 
\sum_{\lambda_i \geq \alpha}
\frac{\alpha}{\lambda_i}  |\scp{e,v_i}|   ^2  
\leq \Cii \sum_{\lambda_i \leq \alpha} 
 \frac{\lambda_i}{\alpha}  |\scp{e,v_i}|   ^2. 
\end{equation} 
\end{definition} 

It is obvious that $\MC_2$ is slightly stronger than $\MC_1$: $\MC_2 \Longrightarrow \MC_1$.
Simplify put, these conditions are  {\em irregularity} conditions on the noise in the 
sense that $e$ should not be smooth (i.e., in the range of $A$). Meanwhile, they 
are quite well understood and are satisfied in many cases. Moreover, it has been shown 
that for mildly ill-posed problems they hold for white and colored noise with probability 
one \cite{KiPePi}.  Although $\MC_2$ is slightly stronger, they are often both satisfied.

Here we show that the error-dependent part of 
$\pls$ is an upper bound for the error propagation term. 
As mentioned before, for this we require a Muckenhoupt condition:

\begin{proposition}
Let $\yd-y$ satisfy a Muckenhoupt-type condition $\MC_2$ with constant $\Cii$. 
Then with $\rho_\alpha(\lambda)$ corresponding to the $\pls$-functional, we have 
\[ \|\xad-\xa\| \leq \sqrt{\Cii+1} \|\rho_\alpha(AA^*)^\frac{1}{2}(\yd-y)\|.  \]
\end{proposition}
\begin{proof} 
As usual, the idea of the proof is to split the spectral decomposition into 
terms involving $\lambda \leq \alpha$ and  $\lambda > \alpha$:  
This works  because of  the estimates 
\begin{equation}\label{ratioest} 
\frac{1}{2} \begin{cases} \frac{1}{\alpha} & \lambda \leq \alpha\\
 \frac{1}{\lambda} & \lambda \geq \alpha \end{cases}\leq \frac{1}{\alpha +\lambda} \leq \begin{cases} \frac{1}{\alpha} & \lambda \leq \alpha,\\
 \frac{1}{\lambda} & \lambda \geq \alpha. \end{cases} 
 \end{equation}
Thus, using \eqref{ratioest} and \eqref{eq:MC2}
 \begin{align*}
 \|\xad-\xa\|^2 &= 
 \sum_i  \frac{\lambda_i}{(\lambda_i +\alpha)^2}    |\scp{e,v_i}|^2   
 \\  &\leq 
  \sum_{\lambda_i \leq \alpha }  \frac{\lambda_i}{\alpha^2}    |\scp{e,v_i}|^2  + 
  \sum_{\lambda_i \geq \alpha }  \frac{1}{\lambda_i}    |\scp{e,v_i}|^2  
  \leq  (1+\Cii )    \sum_{\lambda_i \leq \alpha }  \frac{\lambda_i}{\alpha^2}    |\scp{e,v_i}|^2. 
%
\end{align*} 
Conversely the $\psi$-expression can be estimated  as
\begin{align*}  
&\|\rho_\alpha(AA^*)^\frac{1}{2}(\yd-y)\|^2 = 
  \sum_i \frac{\lambda_i \alpha}{(\lambda_i +\alpha)^3}  |(e,v_i)|^2    \\
 & \qquad \quad =   \sum_{\lambda_i \leq \alpha }   \frac{\lambda_i \alpha}{(\lambda_i +\alpha)^3}
 |\scp{e,v_i}|^2   + 
   \sum_{\lambda_i \geq \alpha }   \frac{\lambda_i \alpha}{\alpha^3}   |\scp{e,v_i}|^2    
\geq   \sum_{\lambda_i \geq \alpha }   \frac{\lambda_i }{\alpha^2}  |\scp{e,v_i}|^2, 
\end{align*}
which yields the statement. 
\end{proof} 
\begin{remark} \rm
Note that the stability part of the 
simple-L curve method behaves similar to the QO-method, for which  also the 
condition $\MC_2$ has been postulated to obtain the analogous estimate. This is different to 
the HD- and HR-methods, where the condition $\MC_1$ is sufficient~\cite{KiETNA}.
\end{remark}

The next step involves  the approximation error: 
\begin{proposition}
Suppose that $\xd \not = 0$ satisfies a source condition \eqref{source}  with 
$\mu \leq 1$. Then for $\alpha \in (0, \alpha_{max})$,
and $\rho_\alpha(\lambda)$ corresponding to the $\pls$-functional, 
there is a constant $C$ such that 
\[ \|\xa-\xd\| \leq \frac{C}{ \|A^* A \xd\|^{2 \mu}}   \|\rho_\alpha(AA^*)^\frac{1}{2}y\|^{2\mu}. \]
\end{proposition}
\begin{proof}
As $(\alpha +\lambda_i) \leq \alpha_{max} + \|A\|^2 =:C_3$, we have that 
\begin{align*}  \|\rho_\alpha(AA^*)^\frac{1}{2}y\|^2   
&= \sum_i 
\frac{\alpha \lambda_i}{(\lambda_i +\alpha)^3}\lambda_i 
|\scp{\xd,u_i}|^2 \\
&\geq  \frac{\alpha}{C_3^3} \sum_i  \lambda_i^2 |\scp{\xd,u_i}|^2 = 
\frac{\alpha \|A^* A \xd\|^2 }{C_3^3 }.
\end{align*}
Conversely, from the classical convergence rate estimate \eqref{eq:conratest} we  obtain 
with a generic constant $C$ that
\begin{align*} 
\|\xa-\xd\|& \leq C \alpha^\mu  \leq C\left( 
\frac{C_3^3}{ \|A^* A \xd\|^2} 
\|\rho_\alpha(AA^*)^\frac{1}{2}y\|^2 \right)^\mu 
\\  &\leq 
\frac{C } { \|A^* A \xd\|^{2 \mu}}  \|\rho_\alpha(AA^*)^\frac{1}{2}y\|^{2 \mu}. 
\end{align*} 
Moreover, we note that $\xd$ is a minimum-norm solution and 
thus in $N(A)^\bot$. Thus if $\xd \not = 0$, then 
$A^* A \xd \not = 0$.
\end{proof} 

If we impose a certain regularity assumption on $\xd$, then 
it can be shown that the approximation part of $\pls$, $\|\rho_\alpha(AA^*)^\frac{1}{2}y\|$, is an upper bound for 
the approximation error. 
The regularity assumption \cite{KiNe,KiETNA} is similar to the Muckenhoupt-type condition 
but with the spectral parts interchanged: 
\begin{equation}\label{regcond} 
\sum_{\lambda_i \leq \alpha }
 |\scp{\xd,u_i}|^2 \leq D \sum_{\lambda_i \geq \alpha }
\frac{\alpha}{\lambda_i}   |\scp{\xd,u_i}|^2.
\end{equation}
For a comparison with other situations, we also state a different regularity condition that is also used in  \cite{KiETNA}:
\begin{equation}\label{regcond2} 
\sum_{\lambda_i \leq \alpha }
 |\scp{\xd,u_i}|^2 \leq D \sum_{\lambda_i \geq \alpha }
\left(\frac{\alpha}{\lambda_i}\right)^2   |\scp{\xd,u_i}|^2.
\end{equation}
Obviously, the first of these conditions, \eqref{regcond}, is weaker and the second 
implies the first. For the simple L-curve method, the weaker one suffices: 


\begin{proposition}
Let $\xd$ satisfy the regularity condition in  \eqref{regcond}.
 Then for $\alpha \in (0, \alpha_{max})$,
and $\rho_\alpha(\lambda)$ corresponding to the $\pls$-functional, there is a constant $C$ such that 
\[ \|\xa-\xd\| \leq  {C} \|\rho_\alpha(AA^*)^\frac{1}{2} y\|. \]
\end{proposition} 
\begin{proof} 
Using the splitting of the sums and  \eqref{ratioest}, we have
\begin{align*} 
&\|\xa-\xd\|^2 = \sum_i \frac{\alpha^2}{(\lambda_i +\alpha)^2}   |\scp {\xd,u_i}|^2 \\
& = \sum_{\lambda_i \leq \alpha }
\frac{\alpha^2}{(\lambda_i +\alpha)^2}  |\scp {\xd,u_i}|^2 + 
\sum_{\lambda_i \geq \alpha }
\frac{\alpha^2}{(\lambda_i +\alpha)^2}   |\scp {\xd,u_i}|^2\\
& \leq \sum_{\lambda_i \leq \alpha }
   |\scp {\xd,u_i}|^2 + \sum_{\lambda_i \geq \alpha }
\frac{\alpha^2}{\lambda_i^2} |\scp {\xd,u_i}|^2 \\ 
& \leq \sum_{\lambda_i \leq \alpha }
   |\scp {\xd,u_i}|^2 + \sum_{\lambda_i \geq \alpha }
\frac{\alpha}{\lambda_i}   |\scp {\xd,u_i}|^2. 
\end{align*} 
While for the approximation part of  $\pls$, using \eqref{ratioest} again, we obtain
\begin{align*} 
&\|\rho_\alpha(AA^*)^\frac{1}{2} y\|^2 = 
\sum_i 
\frac{\alpha \lambda_i}{(\lambda_i +\alpha)^3}\lambda_i 
 |\scp {\xd,u_i}|^2 \\
&\qquad \geq 
\sum_{\lambda_i \leq \alpha }
\frac{\alpha \lambda_i^2}{(\lambda_i +\alpha)^3}   |\scp {\xd,u_i}|^2 + 
\sum_{\lambda_i \geq \alpha }
\frac{\alpha \lambda_i^2}{(\lambda_i +\alpha)^3}   |\scp {\xd,u_i}|^2\\
&\qquad \geq \frac{1}{2} \sum_{\lambda_i \leq \alpha }
\frac{\lambda_i^2}{\alpha^2}   |\scp {\xd,u_i}|^2 + 
\frac{1}{2} \sum_{\lambda_i \geq \alpha }
\frac{\alpha}{\lambda_i}   |\scp {\xd,u_i}|^2\\
&\qquad \geq \frac{1}{2} \sum_{\lambda_i \geq \alpha }
\frac{\alpha}{\lambda_i}   |\scp {\xd,u_i}|^2.
\end{align*} 
Thus, the regularity condition  \eqref{regcond} ensures the bound. 
\end{proof}

Together with Theorem~\ref{conth} and the previous estimates, 
we arrive at the main theorem:
\begin{theorem}\label{th:one}
Let the error satisfy a Muckenhoupt-type condition $\MC_2$,
let $\xd$ satisfy a source condition \eqref{source} with $\mu \leq 1$, 
and let $\|\xd\| \not = 0$. 

 Then choosing the regularization parameter $\as$ as the minimizer of $\pls$ 
  yields the following error bounds 
 \[ \|\xads -\xd\| \leq C  \delta^{\frac{2\tilde{\mu}}{2\tilde{\mu}+1}  2 \tilde{\mu}} , \qquad \tilde{\mu} = \min\{\mu,\frac{1}{2}\}.\] 
 If, moreover, $\xd$ satisfies a regularity condition \eqref{regcond}, then 
 the optimal-order (for $\mu \leq \frac{1}{2}$)  estimate 
  \[ \|\xads -\xd\| \leq C  \delta^{\frac{2\tilde{\mu}}{2\tilde{\mu}+1}}   , \qquad \tilde{\mu} = \min\{\mu,\frac{1}{2}\},\] 
  holds.
\end{theorem}

\begin{remark} \rm
The convergence theorem for the simple-L method should be compared to the corresponding 
results for the HD, HR, and QO-rules in \cite{KiETNA}: Essentially, the functional 
$\pls$ requires the same conditions as the QO-rule, but it only achieves the 
optimal order (in the best case when a regularity condition holds) 
up to $\mu \leq \frac{1}{2}$, while the QO-rule 
does this (under the same regularity condition) 
for all $\mu$ up to the saturation index $\mu = 1$. In this sense, 
the QO-rule is an improvement of the simple-L method.  This is similar to the 
relations between HD and HR: the heuristic discrepancy method, $\psi_{HD}$, 
can also be only optimal up to $\mu \leq \frac{1}{2}$, while the Hanke-Raus method 
improves this up to $\mu = 1$. Thus $\pls$ is related to $\psi_{QO}$ in a similar way 
to how  $\psi_{HD}$ is related to  $\psi_{HR}$. 
\end{remark}

\subsection{\boldmath Convergence for $\plr$}
The previous analysis can be extended to the simple-L ratio method. 
We now consider a functional of the form 
\begin{equation}\label{combfun} 
\psi(\alpha,\yd) = \rhi(\alpha)\pls(\alpha,\yd), 
\end{equation} 
where $\rhi$ is a nonnegative function. The simple-L ratio corresponds to $\rhi(\alpha) = \frac{1}{\|\xad\|}$: 
We have the following proposition. (Here, $\Id$ denotes the identity function $x \to x$).
\begin{proposition}\label{prop:prod}
Let the error satisfy a Muckenhoupt-type condition $\MC_2$,
and let $\eqref{estd}$ hold for $\rho_\alpha$ corresponding to $\pls$.
Suppose that $\as$ is selected by 
\eqref{combfun}.
Then the following error estimates hold: For $\bar{\alpha} \in (0,\alpha_{max})$ arbitrary 
\begin{equation}
\begin{split}
 \|\xads -\xd\|  &\leq \frac{\rhi(\ba)}{\rhi(\as)} \left(\uppera(\ba) + \upperd(\ba) \right) + 
2  \max\{\Phi,\Car \Id\}(\uppera(\ba))  \qquad \text{ if } \as \leq \ba, \\ 
 \|\xads -\xd\|   &\leq \Car \upperd(\ba) + \Phi\left[\upperd(\ba) + \frac{\rhi(\ba)}{\rhi(\as)} \left(\upperd(\ba) + \uppera(\ba) \right)
 \right]   \qquad \text{ if } \as \geq \ba. 
 \end{split} 
\end{equation}
 Here $\upperd$ and $\uppera$ are defined in \eqref{eq:slup} and \eqref{eq:slap}.
\end{proposition} 
\begin{proof}
Let $\as \leq \ba$. Then  from the previous estimates for $\pls$, 
the minimization property of $\rhi(\alpha) \pls(\alpha,\yd)$, 
and by the monotonicity of $\uppera$, we have 
\begin{align*} 
\|\xas-\xd\| &\leq \Phi\left(\pls(\as,\xd)\right) \leq \Phi(\uppera(\as)) \leq  \Phi(\uppera(\ba)),  \\ 
\rhi(\as) \|\xads-\xas\| &\leq  \Car \rhi(\as) \pls(\as,e) \\
& \leq \Car \rhi(\as) \pls(\as,\yd) + \Car \rhi(\as) \pls(\as,\xd)\\
&\leq 
\Car \rhi(\ba) \psi(\ba,\yd) +  \Car \rhi(\as) \uppera(\ba)\\
&\leq \Car \rhi(\ba)  \uppera(\ba) + \Car \rhi(\ba) \upperd(\ba) + \Car \rhi(\as) \uppera(\ba).
\end{align*} 
For $\as \geq \ba$, with the same arguments and  from  the monotonicity of $\upperd$ that
\begin{align*} 
\|\xads-\xas\| &\leq  \Car \pls(\as,e) \leq \Car \upperd(\as) \leq C_1\upperd(\ba), \\ 
\Phi^{-1}(\|\xa-\xd\|) &\leq \rhi(\as)\frac{\pls(\as,\yd)}{\rhi(\as)} + \pls(\as,e) 
\leq 
\rhi(\ba)\frac{\pls(\ba,\yd)}{\rhi(\as)} + \upperd(\as)  \\ 
& \leq  \frac{\rhi(\ba)}{\rhi(\as)} \uppera(\ba) +   (\frac{\rhi(\ba)}{\rhi(\as)}+1) \upperd(\ba). 
  \end{align*} 
\end{proof}

\begin{theorem}
Under the same conditions as Theorem~\ref{th:one} and if $\as$ is chosen 
by the simple-L ratio-method,
then the same error bounds hold if  $\delta$ is sufficiently small. 
\end{theorem} 
\begin{proof}
We have that $\Phi(x)  = x^\xi$, where $\xi \leq 1$. 
The error estimates can be rewritten as 
\begin{align*}
\rhi(\as) \|\xads -\xd\|  &\leq {\rhi(\ba)} \left(\uppera(\ba) + \upperd(\ba) \right) + 
\rhi(\as)  2  \max\{\Phi,C_1 \Id\}(\uppera(\ba))&  &\text{ if } \as \leq \ba, \\ 
\rhi(\as)  \|\xads -\xd\|   &\leq C_1 \rhi(\as) \upperd(\ba) \\
&+ 
\Phi\left[ \rhi(\as)^\frac{1}{\xi} \upperd(\ba) + \rhi(\as)^{\frac{1}{\xi}-1} {\rhi(\ba)} \left(\upperd(\ba) + \uppera(\ba) \right)
 \right]&    &\text{ if } \as \geq \ba.  
\end{align*}

For the simple-L ratio method, we have $\rhi(\alpha) = \frac{1}{\|\xad\|}$. We take $\ba$ as the 
optimal order choice  $\ba \sim \delta^\frac{2}{2\tilde{\mu}+1}$, which implies 
$\xad \to \xd$, and hence for $\delta$ sufficiently small, we have that 
$\rhi(\ba) \sim  \frac{1}{\|\xd \|}$.  From the standard theory it follows that 
$\|\xad\|$ is monotonically decreasing, hence $\rhi(\alpha)$ is monotonically increasing.  
Thus, with some constant $C$ 
\[ \rhi(\as) \leq \rhi(\alpha_{max}) \leq C. \] 
In any case, the expressions 
$\rhi(\as)^\frac{1}{\xi},$ ${\rhi(\ba)}$ and, as $\xi \leq 1$, also  
$\rhi(\as)^{\frac{1}{\xi}-1}$ stay bounded.  
Hence, we obtain  that 
\[ \rhi(\as) \|\xads-\xd\| \leq C'\max\{\Phi,\Id\}\left(C \delta^\frac{2\tilde{\mu}}{2\tilde{\mu}+1}\right), \]
with different constants $C,C'$. Moreover, since 
\[ \rhi(\as) = \frac{1}{\|\xads\|} \geq \frac{1}{\|\xads -\xd\| + \|\xd\|}, \]
we have that 
\[  \frac{ \|\xads-\xd\|}{\|\xads -\xd\| + \|\xd\|} \leq  C'\max\{\Phi,\Id\}\left(C \delta^\frac{2\tilde{\mu}}{2\tilde{\mu}+1}\right). \]
Since $\frac{x}{x +\|\xd\|} \sim x$ for $x$ small, this yields the same order of estimates as before. 
\end{proof} 

The reason for requiring that $\delta$ is small is because the expression 
$\frac{ \|\xads-\xd\|}{\|\xads -\xd\| + \|\xd\|}$ is bounded by 1. Hence if the right-hand side 
(which is of the order of the optimal convergence) is large, the estimate holds trivially true but is void of content.

\subsection{Extension to other regularization methods}
We note that the simplification of the curvature of the 
L-curve relies heavily on Tikhonov regularization, which is the only regularization method for which 
formula \eqref{form} holds true. For general 
regularization schemes, the expression for the curvature 
becomes rather complicated.

With the same definition of the L-curve, 
the curvature can be calculated to be 
\[ \cur = \frac{\rho \eta}{\left( {\rho'}^2 \eta^2 + \rho^2 {\eta'}^2 \right)^\frac{3}{2}} 
\left( \eta'' \eta \rho ' \rho - \rho '' \rho \eta' \eta - \eta'^2 \rho' \rho + \rho'^2 \eta' \eta\right). \] 
For Tikhonov regularization, this can be simplified by the formula $\rho' = -\alpha \eta'$, 
but for other regularization methods, this is no longer possible. 
Similar as above, however, we introduce the variable $\zeta = \frac{\rho \eta'}{\rho'\eta}$, which 
for Tikhonov regularization agrees with the definition given in Lemma~\ref{firstlem}. 
Then we obtain that 
\begin{align}  \cur &= \frac{1}{|{\rho'} \eta|^3 \left( 1 + \zeta^2 \right)^\frac{3}{2}} 
\left( \eta'' \eta^2 \rho ' \rho^2 - \rho '' \rho^2 \eta' \eta^2 -  \zeta^2 ({\rho'} \eta)^3 + \zeta  ({\rho'} \eta)^3\right) \nonumber \\ 
& = 
\left[\frac{\eta \eta''}{\eta'^2} -\frac{\rho'' \eta}{\eta' \rho'} \right]  \frac{\zeta^2}{ \left( 1 + \zeta^2 \right)^\frac{3}{2}} +
\frac{-\zeta^2 + \zeta}{ \left( 1 + \zeta^2 \right)^\frac{3}{2}} . \label{expr}
\end{align} 
For Tikhonov regularization, the identity $\rho' = -\alpha \eta'$ yields that 
\[ \rho'' = - \eta' - \alpha \eta'' =  - \eta'  + \frac{\rho'}{\eta'} \eta'', \]
and this yields the formula~\eqref{formsi}. Thus, a fully analogous functional 
corresponding to $\plr$ would be to minimize the reciprocal of the 
expression in brackets in \eqref{expr}. However, due to the subsequent existence of several 
second-derivative terms, such a method  would not be qualified then to be named  ``simple''. 

We try to simplify the expression for asymptotic regularization (cf.~\cite{EnHaNe96}), 
which is a continuous version of classical Landweber iteration. 
The method is defined via an initial value problem in Hilbert spaces, 
\[ x(t)' = A^* p(t), \qquad x(0) = 0, \quad t\geq 0,  \]
where $p(t) =\yd - A x(t)$. 
The regularized solution is given by 
\[ \xad = x(\tfrac{1}{\alpha}). \]
When the derivative $x'(t)$ is replaced by a forward difference, this 
yields exactly Landweber iteration. 

For this method, we have the identities 
\[ p' = - A x', \quad x' = A^* p \quad \Rightarrow \pad' = - A A^* p.   \]
Thus, 
\begin{align*}  \eta' &= 2\scp{x,x'} = 2 \scp{A x, p}, \\
 \rho' &= 2 \scp{p,p'} = -2 \scp{p,A x'} = -2 \scp{x',x'} = -2 \|A^*p\|^2, \\
 \eta'' &= 2 \scp{A x', p} + 2\scp{A x, p'} = - \rho' - 2 \scp{Ax, Ax'}, \\
\rho'' &= -4 \scp{A^*p,A^*p'} = -4 \scp{A^*p,A^*p'} = 4\scp{A^*p,A^*A A^*p}. 
\end{align*} 
As the curvature is independent of the parameterization, we may use 
the variable $t$ in place of $\alpha$ to calculate it. Hence, 
the expression in brackets in \eqref{expr} can then be written as 
\begin{align*}
& \frac{\eta}{\eta'}\left( \frac{\eta''}{\eta'} - \frac{\rho'' }{\rho'}\right) 
 = \frac{\|x\|^2}{2\scp{A x, p}}\left( \frac{ \|A^*p\|^2}{\scp{A x, p}} - \frac{\scp{A x, p'}}{\scp{A x, p}} - \frac{2\scp{A^*p,A^*A A^*p}}{\|A^*p\|^2} \right).
\end{align*}
The last expression  $\frac{2\scp{A^*p,A^*A A^*p}}{\|A^*p\|^2}$ is bounded by $\|A^*A\|$.  Thus, the only way  that the L-curve  can have 
a large curvature is when $\scp{A x, p}= \eta'$ is small. This  essentially leads again to the simple L-curve method with the minor difference that 
the derivative is taken with respect to the $t$-variable. 

By analogy, we may transfer these results to Landweber iteration, where derivatives are replaced by finite differences. 
The simple L-curve method would then be defined by minimizing 
\begin{equation}\label{eq:lwthis} \psi(k) = \scp{A x_k, \yd - A x_k} \sim  \scp{x_k, x_{k+1} - x_k}, \end{equation}
over the iteration indices $k$. Clearly, this can be considered a discrete variant of $\pls$, where 
the derivative $\alpha \frac{\partial}{\partial \alpha}$ is replaced by a finite difference. 
Another possibility for defining a simple L-curve method is to use 
\eqref{xy} for general regularization method via their filter functions. In case of Landweber 
iteration this leads to a similar functional as in \eqref{eq:lwthis}, namely
\[ \psi(k) =  \scp{x_k, x_{2k} - x_k}. \]

Of further special interest is to use these methods for nonlinear (e.g., convex) Tikhonov regularization, 
where $\xad$ is defined as minimizer of 
\begin{equation} 
\label{convexfunctional}
x \to \|A x -\yd\|^2 + \alpha R(x), \end{equation}
with a general convex regularization functional $R$. For an analysis of several heuristic rules in this 
context, see \cite{KR}. 
Note that the L-curve method 
is then defined by analogy as a plot of $(\log(R(\xad)), \log(\|A \xad -\yd\|)$. 
It  has been 
applied with success  in  such a context, e.g., in \cite{xuli}. 
One should be cautioned, however, that 
here 
it  is not necessarily true that $\xad$ is differentiable with respect to $\alpha$, and moreover, 
$R(\xad)$ can be $0$, hence the L-graph in its logarithmic form 
is not defined there. 
If $R$ is smooth, then the formula \eqref{form} still holds with 
$\eta(\alpha) = R(\alpha)$, and we may define  a simple-L method as minimization of 
\[ \pls(\alpha) = -\alpha \frac{\partial}{\partial \alpha} R(\xad ). \]
However, for convex Tikhonov regularization it is preferable---due to a possible lack of 
differentiabi\-li\-ty---to replace the derivative $\alpha \frac{\partial}{\partial \alpha}$ by alternative 
expressions. One way  is to use a finite difference approximation on a logarithmic grid yielding 
\begin{equation}\label{discretesimpleLfunctional} 
\pls(\alpha) = R(x_{\alpha_{n+1},\delta} )- R(x_{\alpha_{n},\delta} ), \quad \alpha = \alpha_0 q^n \quad  q < 1.
\end{equation}
Another way is to replace the derivative by expressions obtained by Bregman iteration. 
In this case, the functional would be 
\begin{equation} \pls(\alpha) = R(\xad^{II} )- R(\xad), \label{convexsimplelfunc} \end{equation}
where $\xad^{II}$ is the second Bregman iterate; cf.~\cite{KR}.  Both methods can also be understood as a kind of 
quasi-optimality method, where the ``strict metric'' $d(x,y) = |R(x)- R(y)|$ (cf.~\cite{FLR}) is used for measuring convergence.
(a similar method has been tested in \cite{KLR}). Note that we may similarly adapt the simple-L ratio functional as
\begin{equation}
    \plr = \frac{R({x^\delta_\alpha}^{II})-R(x^\delta_\alpha)}{R(x^\delta_\alpha)},
    \label{convexsimplelratiofunc}
\end{equation}
with the notation as before.

\section{Numerical Tests}
We perform some numerical tests of the proposed methods. 
The noise-level $\delta$ is chosen such that the relative error has 
the values $(0.01\%$, $0.1\%$, $1\%$, $5\%$, $10\%$, $20\%$, $50\%)$.
Here, the first two are classified as "small", the second pair as "medium" and the last triple is classified as "large".
For each noise-level, 
we performed 10 experiments.  We tested the method $\pls$ (simple-L) , $\plr$ (simple-L ratio), the QO-method, 
and the original L-curve method  defined by maximizing the curvature.

A general observation was that whenever the L-curve showed a clear corner, then 
the selected parameter by both $\pls$ and $\plr$  was very close to that corner, 
which confirms the idea of those methods being simplifications of the L-curve method. 
Note, however, that closeness on the L-curve does not necessarily mean that 
the selected parameter is close as well since the parameterization around 
the corner becomes ``slow''. 

We compare the four methods, namely, 
the two new simple-L rules, the QO-rule, and the original 
L-curve,  
according to their total error for  
the respective selected $\alpha$ and calculate the 
ratio of the obtain error to the best possible error: 
\begin{equation}\label{defrat} J(\alpha^\ast) := \frac{d(\xads,\xd)}{\inf_{\alpha} d(\xad,\xd)},\end{equation}
where one would typically compute $J$ with $d(x,y):=\|x-y\|$ for the case of linear regularization.

\subsection{Linear Tikhonov Regularization}
We begin with classical Tikhonov regularization, in which case we compute the regularized solution as \eqref{Tikreg}.
\subsubsection{Diagonal Operator}
At first we consider a diagonal operator $A$ with singular values having 
polynomial decay: $\sigma_i = i^{-s}$ for some value $s$ and 
consider an exact solution also with polynomial decay 
$(\xd,u_i) = (-1)^i i^{-p}$. Furthermore we added random noise 
$\scp{e_i,v_i}  = \delta i^{-0.6}{ \tilde{e}}_i$, where $\tilde{e}_i$ are standard normally distributed 
values.

Table~\ref{mytable} displays the median of the values of 
$J$ over 10 experiments with different random noise realizations and for varying  smoothness indices $\mu$.
\begin{table}
\caption{Tikhonov Regularization, Diagonal Operator: Median of Ratio \eqref{defrat} of errors  rules over 
10 runs.}\label{mytable}
\begin{center}
\begin{tabular}{l|c|c|c|c}
& \multicolumn{1}{c}{simple-L}  & \multicolumn{1}{c}{simple-L rat.} & \multicolumn{1}{c}{QO}& 
\multicolumn{1}{c}{L-curve}  \\ \hline
\multicolumn{1}{l|}{$s = 2$, $\mu = 0.25$} & \multicolumn{4}{c}{} \\ \hline
$\delta$ small & 1.02  & 1.02 & 1.03 &    9.49 \\
$\delta$ medium & 1.01  & 1.02 & 1.08 &    1.78 \\
$\delta$ large & 1.79  & 1.06 & 1.18 &    1.15 \\
$\delta=50$\% & 1.97  & 3.64 & 1.42 &    1.46 \\ \hline
%
\multicolumn{1}{l|}{$s = 2$, $\mu = 0.5$} & \multicolumn{4}{c}{} \\ \hline
$\delta$ small & 1.48  & 1.48 & 1.01 &    50.68 \\
$\delta$ medium & 1.66  & 1.72 & 1.07 &    3.78 \\
$\delta$ large & 1.78  & 1.59 & 1.01 &    2.52 \\
$\delta=50$\% & 3.09  & 5.07 & 1.48 &    1.90 \\ \hline
\multicolumn{1}{l|}{$s = 2$, $\mu =1$} & \multicolumn{4}{c}{} \\ \hline
$\delta$ small & 3.88  & 3.88 & 1.07 &    77.12 \\
$\delta$ medium & 2.01  & 2.01 & 1.07 &    7.98 \\
$\delta$ large & 1.57  & 1.66 & 1.08 &    2.33 \\
$\delta=50$\% & 2.97  & 4.07 & 1.27 &    1.32 \\
\end{tabular}
\end{center}
\end{table} 
The table provides some information about the performance
of the rules. Based on additional numbers not presented here, we can state some conclusions: 
\begin{itemize}
    \item The simple-L and simple-L ratio outperform 
    the other rules for small smoothness index $\mu = 0.25$ and small data noise. 
    Except for very large $\delta$,
    the simple-L ratio is slightly better than the simple-L curve. For very large $\delta$, the simple-L method works but is inferior to QO while the simple-L ratio method fails 
    then.
    \item For high smoothness index, the QO-rule outperforms the other rules and it is the method of choice then.
    \item The original L-curve method often fails for 
    small $\delta$. For larger $\delta$ it works often only 
    acceptably. 
    Only in situations when $\delta$
    is quite large ($>20\%$) did we find several 
    instances when it outperforms all other rules. 
\end{itemize}
A similar experiment was performed for a more smoothing 
operator by setting $s = 4$ with similar conclusions. 
We note that the theory has indicated that 
for $\mu=0.5$, the simple-L curve is order optimal 
without any additional condition on $\xd$ while for 
the QO-rule this happens at $\mu =1$. One would thus 
expect that the simple-L rule perform better for 
$\mu=0.5$. However, this was not the case (only for $\mu \leq 0.25)$ and the reason 
is unclear. (We did not do experiments with 
an $\xd$ that does not satisfy the regularity condition \eqref{regcond}, though). 
Still, the result that the simple-L methods perform 
better for small $\mu$ is backed by the numerical results.

\subsubsection{Examples from Regularization Tools}
For the next scenario, we consider the  tomography (i.e. \texttt{tomo}) operator from Hansen's Regularization Tools \cite{HaReg} and seek to reconstruct the solution provided in the package. The data is corrupted with normally distributed random noise as before, i.e., $e_i=\delta\tilde{e}_i$. Note that the operator and solution are normalised such that $\|A\|=\|x^\dagger\|=1$ and our parameter search is restricted to the interval $[\sigma_{\text{min}},\|A\|^2]$, where $\sigma_{\text{min}}$ is the smallest singular value of the operator $A^\ast A$.
Similarly as for the previous experiment, in Table~\ref{TikhTomoTable}, we record the median of the values of $J$ over 10 different experiments with varying random noise realizations. 

\begin{table}
\caption{Tikhonov Regularization, \texttt{tomo} Operator: Median of Ratio \eqref{defrat} of errors  rules over 
10 runs.}\label{TikhTomoTable}
\begin{center}
\begin{tabular}{l|c|c|c|c}
& \multicolumn{1}{c}{simple-L}  & \multicolumn{1}{c}{simple-L rat.} & \multicolumn{1}{c}{QO}& 
\multicolumn{1}{c}{L-curve}  \\ \hline
$\delta$ small & 1.24   & 1.24  & 1.40  & 15.02     \\
$\delta$ medium & 1.04  & 1.04 & 1.01 & 2.09    \\
$\delta$ large & 1.76  & 1.31  & 1.76 & 1.81    \\
$\delta=50$\% & 1.23  & 1.23  & 1.23  & 1.27    \\
\end{tabular}
\end{center}
\end{table}

Next we consider the \texttt{heat} operator from Hansen's Regularization Tools with identical setup as before, except we choose $\alpha_{\text{min}}=10^{-9}$ as a fixed lower bound, since the singular values for the heat operator decay much faster, thus selecting $\alpha_{\text{min}}$ as the smallest singular value would be far too unstable.
In Table~\ref{TikhHeatTable}, one can find a record of the median values of $J$ for 10 different realizations of each noise level:

\begin{table}
\caption{Tikhonov Regularization, \texttt{heat} Operator: Median of Ratio \eqref{defrat} of errors  rules over 
10 runs.}\label{TikhHeatTable}
\begin{center}
\begin{tabular}{l|c|c|c|c}
& \multicolumn{1}{c}{simple-L}  & \multicolumn{1}{c}{simple-L rat.} & \multicolumn{1}{c}{QO}& 
\multicolumn{1}{c}{L-curve}  \\ \hline
$\delta$ small & 1.33  & 1.33  & 1.02  & 2.20  \\
$\delta$ medium & 1.23  & 1.25 & 1.02 & 1.11    \\
$\delta$ large & 1.61  & 1.07  & 1.45 & 1.36    \\
$\delta=50$\% & 1.85  & 1.19  & 1.37  & 1.41    \\
\end{tabular}
\end{center}
\end{table}
Overall, the RegTool examples indicate a similar behaviour as before, 
with the simple-L rules being competitive for small noise and sometimes 
even outperforming the  QO-rule, which is in general hard to beat. The original L-curve 
is particularly successful for large noise but often seems to fail in the other cases. 

\subsection{Convex Tikhonov Regularization}
We now investigate the heuristic rules for convex Tikhonov regularization, i.e., we consider $x^\delta_\alpha$ as the minimizer of the functional \eqref{convexfunctional} with 
a nonquadratic penalty $R$. Note that the convergence theory of the present paper 
does not cover this case. For the HD, HR, and QO-rules, some convergence results of the 
theory in \cite{KiETNA} have been extended to the convex case in \cite{KR}. 

Henceforth, the simple-L methods will consist of minimizing the functionals \eqref{convexfunctional} and \eqref{convexsimplelratiofunc}. Note that we did consider \eqref{discretesimpleLfunctional} as an alternative "convexification" of the simple L-curve method, but the former method appeared to yield more fruitful results and we therefore opted to stick with that.

\subsubsection{\boldmath $\ell^1$ Regularization}
To begin with, we consider $R=\|\cdot\|_{1}$ and the tomography operator \texttt{tomo}  as before,  
 but this time we would like to reconstruct a sparse solution $x^\dagger$.
Note that we compute a minimizer via FISTA \cite{BeTe}. In this case,  we measure the error with the $\ell^1$ norm, i.e., we compute $J$ with $d(x,y):=\|x-y\|_1$. 

In our experiments, we observed that the values of the aforementioned simple-L functionals were particularly small, therefore on occasion yielding negative values due to numerical errors. This problem was easily rectified however by taking the absolute value of \eqref{convexfunctional} and \eqref{convexsimplelratiofunc}, respectively, which is theoretically equivalent to the original functionals in any case. For the quasi-optimality functional, one now has several possible options, but we opted to use 
\begin{equation}
    \psi_{QO}(\alpha)=D_{\xi^\delta_\alpha}({x^\delta_\alpha}^{II},x^\delta_\alpha),
    \label{convexqofunc}
\end{equation}
the Bregman distance of the second Bregman iterate and the Tikhonov solution in the direction $\xi^\delta_\alpha\in\partial R(x^\delta_\alpha)$, which is the so-called right quasi-optimality rule discussed in \cite{KR}. For selecting the parameter according to the L-curve method of Hansen, maximizing the curvature via \eqref{form} is no longer an implementable strategy as $R$ is now non-smooth. Therefore, we elected to choose the parameter by visually inspecting the graph ($\log\|Ax^\delta_\alpha-y_\delta\|,\log\|x^\delta_\alpha\|_1)$ and selecting the appropriate corner point manually.
In Table~\ref{ell1TomoTable}, one may find a recording of the results.
\begin{table}
\caption{$\ell^1$ Regularization, \texttt{tomo} Operator: Median of Ratio \eqref{defrat} of errors  rules over 
10 runs.}\label{ell1TomoTable}
\begin{center}
\begin{tabular}{l|c|c|c|c}
& \multicolumn{1}{c}{simple-L}  & \multicolumn{1}{c}{simple-L rat.} & \multicolumn{1}{c}{QO}& 
\multicolumn{1}{c}{L-curve}  \\ \hline
$\delta$ small & 2.97   & 2.97   &  9.51 & 88.11   \\
$\delta$ medium & 60.48  &  81.93 & 5.87  & 7.82    \\
$\delta$ large & 1.15  & 1.15  & 1.01  & 57.98   \\
$\delta=50$\% & 1.31  & 1.14  & 1.31  & 51.22    \\
\end{tabular}
\end{center}
\end{table}

We note the following observations:
\begin{itemize}
    \item As mentioned already, the simple-L functionals produced very small values and therefore were somewhat oscillatory, i.e., they were prone to exhibiting multiple local minima. Our algorithm selected the smallest interior minimum, but in some plots, we observed that there were larger local minima which would have corresponded to a more accurate estimation of the optimal parameter.
    \item In order to visually detect the corner of the L-curve, it should be noted that one had to magnify the graph. For large noise levels, there was no such corner point.
\end{itemize}

\subsection{$\ell^{\frac{3}{2}}$ Regularization}
Continuing with the theme of convex Tikhonov regularization and more specifically $\ell^p$ regularization, we now consider \eqref{convexfunctional} with $R=\|\cdot\|_p$ and $p=\frac{3}{2}$. The considered forward operator $A:\ell^p(\mathbb{N})\to\ell^2(\mathbb{N})$ is a diagonal operator with polynomially decaying singular values as considered previously i.e., $\sigma_i = i^{-s}$ and we also consider a solution with polynomial decay 
$\scp{\xd,u_i} = (-1)^i i^{-p}$ and add random noise
$\scp{e_i,v_i}  = \delta i^{-0.6}{ \tilde{e}}_i$. Note that in this scenario, we are easily able to compute the Tikhonov solution and second Bregman iterate as we have a closed form solution of the associated proximal mapping operator; see~\cite{KR}. 

\begin{table}
\caption{$\ell^p$ Regularization, Diagonal Operator: Median of Ratio \eqref{defrat} of errors  rules over 
10 runs.}\label{convexdiagtable}
\begin{center}
\begin{tabular}{l|c|c|c|c}
& \multicolumn{1}{c}{simple-L}  & \multicolumn{1}{c}{simple-L rat.} & \multicolumn{1}{c}{QO}& 
\multicolumn{1}{c}{L-curve}  \\ \hline
\multicolumn{1}{l|}{$s = 2$, $\mu = 0.25$} & \multicolumn{4}{c}{} \\ \hline
$\delta$ small & 6.59  & 6.59 & 1.02 &  471.03   \\
$\delta$ medium & 1.95  & 1.95 & 1.18 & 31.22    \\
$\delta$ large & 1.10  & 1.10 & 1.07 &  1.11   \\
$\delta=50$\% & 1.13  & 1.21 & 1.11 &  1.22  \\ \hline
\multicolumn{1}{l|}{$s = 2$, $\mu = 0.5$} & \multicolumn{4}{c}{} \\ \hline
$\delta$ small & 14.41  & 14.41 & 1.00 &  8.91   \\
$\delta$ medium & 2.05  & 2.05 & 1.01 & 115.00    \\
$\delta$ large & 1.09  & 1.09 & 1.03 &  1.15   \\
$\delta=50$\% & 4.72  & 5.61 & 1.01 &  1.80   \\ \hline
\multicolumn{1}{l|}{$s = 2$, $\mu =1$} & \multicolumn{4}{c}{} \\ \hline
$\delta$ small & 20.51  & 20.51 & 1.46 &  4.29   \\
$\delta$ medium & 1.36  & 1.36 & 1.33 & 107.77    \\
$\delta$ large & 1.14  & 1.14 & 1.40 &   1.34  \\
$\delta=50$\% & 7.06  & 9.28 & 1.08 &  1.53  \\
\end{tabular}
\end{center}
\end{table} 

A table of results is compiled in Table~\ref{convexdiagtable} and the following observations are noted:
\begin{itemize}
    \item Barring the quasi-optimality rule, all methods were generally subpar in case of small noise for all tested smoothness indices. In general, the quasi-optimality rule would appear to be the best performing overall at least, although trumped on a few occasions.
    \item The "sweet spot" for both simple-L methods appears to be medium to large noise. Overall, at least, they appear to perform marginally better for smaller smoothness indices. The original L-curve method performs quite well for larger noise, as has been observed in other experiments, but the margin for error is quite large for smaller noise levels.
\end{itemize}

\subsection{TV Regularization}
We now suppose that $x^\delta_\alpha$ is the minimizer of \eqref{convexfunctional} with $R=|.|_{TV}$ 
the total variation seminorm. Note that for numerical implementation, the above functional is often discretized as $R(x)=\sum\|\nabla x\|_1$, with $\nabla$ denoting a (e.g., forward) difference operator. The functional is minimized using FISTA with the proximal mapping operator for the total variation seminorm being computed by a fast Newton-type method as in \cite{KR}. In this case, we compute the error with respect to $\alpha$ via the so-called strict metric
\[
d_{\text{strict}}(x^\delta_\alpha,x^\dagger):=|R(x^\delta_\alpha)-R(x^\dagger)|+\|x^\delta_\alpha-x^\dagger\|_1,
\]
which was suggested in, e.g., \cite{KLR}, and we subsequently record the values of $J$ with $d=d_{\text{strict}}$, the results of which are provided in Table~\ref{TVTomoTable}.
\begin{table}
\caption{TV Regularization, \texttt{tomo} Operator: Median of Ratio \eqref{defrat} of errors  rules over 
10 runs.}\label{TVTomoTable}
\begin{center}
\begin{tabular}{l|c|c|c|c}
& \multicolumn{1}{c}{simple-L}  & \multicolumn{1}{c}{simple-L rat.} & \multicolumn{1}{c}{QO}& 
\multicolumn{1}{c}{L-curve}  \\ \hline
$\delta$ small & 8.67    & 8.67    &  2.21  & 22.88    \\
$\delta$ medium &  9.46 & 5.79  & 4.73  & 8.67    \\
$\delta$ large & 1.44  & 1.07  & 1.41  & 8.15    \\
$\delta=50$\% & 1.50  & 1.06  & 1.50  & 10.90    \\
\end{tabular}
\end{center}
\end{table}
We note the following observations:
\begin{itemize}
    \item The graph for the L-curve appeared only to produce an "L" shape for smaller noise.
    \item All rules appear to be suboptimal for small and medium noise, with the quasi-optimality rule faring slightly better than the other rules in that case.
    \item For larger noise levels, the simple-L ratio method is clearly the best performing.
\end{itemize}

\subsection{Summary}
To summarize the numerical results presented above, the simple-L methods are near optimal for linear Tikhonov regularization in case of low smoothness of the exact 
solution. 
Moreover, the simple-L rule in particular edges the simple-L ratio rule, but the margin of difference is small and only apparent for larger noise levels. 

We also considered convex Tikhonov regularization for which the simple-L functionals had to be adapted from their original forms. In any case, they were successfully implemented and demonstrated above satisfactory results. Interesting to note however, was that in this setting, the simple-L ratio method appeared to present itself as the slightly superior of the two variants.  

The original L-curve method of Hansen appears to have problems in case of 
small noise levels but is a reasonable choice for linear Tikhonov regularization 
and large noise. 

\section{Conclusion}
In conclusion, we reduced the standard L-curve method for parameter selection to a minimization problem of an error estimating surrogate functional from which two new parameter choice rules were born: the simple-L and simple-L ratio methods. The rules yielded convergence rates for Tikhonov regularization under a Muckenhout-type condition $\MC_2$, akin to that required for the quasi-optimality rule, but saturate early like the heuristic discrepancy rule.

The subsequent numerical experiments furthermore verified that the simple-L methods are not only capable of substituting as parameter choice rules for the L-curve method, but also outperform it the majority of the time, performing similarly even to the quasi-optimality rule, whilst being much easier to implement than the original L-curve method.
 

\bibliographystyle{siam}
\bibliography{lbib}
\end{document}